\documentclass[12pt]{amsproc}
\newtheorem{theorem}{\sc Theorem}[section]
\newtheorem{lemma}[theorem]{\sc Lemma}

\begin{document}
\title{Almost Engel linear groups }
\author{Pavel Shumyatsky }
\address{ Department of Mathematics, University of Brasilia,
Brasilia-DF, 70910-900 Brazil}
\email{pavel@unb.br}
\thanks{This research was supported by FAPDF and CNPq-Brazil}
\begin{abstract}
A group $G$ is almost Engel if for every $g\in G$ there is a finite set ${\mathcal E}(g)$ such that for every $x\in G$ all sufficiently long commutators $[x,{}_n\, g]$ belong to ${\mathcal E}(g)$, that is, for every $x\in G$ there is a positive integer $n(x,g)$ such that $[x,{}_n\, g]\in {\mathcal E}(g)$ whenever $n(x,g)\leq n$. A group $G$ is almost nil if it is almost Engel and for every $g\in G$ there is a positive integer $n$ such that $[x,{}_n\, g]\in{\mathcal E}(g)$ for every $x\in G$.

We prove that if a linear group $G$ is almost Engel, then $G$ is finite-by-hypercentral. If $G$ is almost nil, then $G$ is finite-by-nilpotent.
\end{abstract}

\maketitle

\section{Introduction} 

By a linear group we understand here a subgroup of $GL(m,F)$ for some field $F$ and a positive integer $m$. An element $g$ of a group $G$ is called a (left) Engel element if for any $x\in G$ there exists $n=n(x,g)\geq 1$ such that $[x,{}_n\, g]=1$. As usual, the commutator $[x,{}_n\, g]$ is defined recursively by the rule $$[x,{}_n\, g]=[[x,{}_{n-1}\, g],g]$$ assuming $[x,{}_0\, g]=x$. If $n$ can be chosen independently of $x$, then $g$ is a {\em (left) $n$-Engel element}. A group $G$ is called Engel if all elements of $G$ are Engel. It is called $n$-Engel if all its elements are $n$-Engel. A group is said to be locally nilpotent if every finite subset generates a nilpotent subgroup. Clearly, any locally nilpotent group is an Engel group. It is a long-standing problem whether any $n$-Engel group is locally nilpotent. Engel linear groups are known to be locally nilpotent (cf. \cite{gara, grue}).

We say that a group $G$ is almost Engel if for every $g\in G$ there is a finite set ${\mathcal E}(g)$ such that for every $x\in G$ all sufficiently long commutators $[x,{}_n\, g]$ belong to ${\mathcal E}(g)$, that is, for every $x\in G$ there is a positive integer $n(x,g)$ such that $[x,{}_n\, g]\in {\mathcal E}(g)$ whenever $n(x,g)\leq n$. (Thus, Engel groups are precisely the almost Engel groups for which we can choose ${\mathcal E}(g)=\{ 1\}$ for all $g\in G$.) We say that a group $G$ is nil if for every $g\in G$ there is a positive integer $n$ depending on $g$ such that $g$ is $n$-Engel. The group $G$ will be called almost nil if it is almost Engel and for every $g\in G$ there is a positive integer $n$ depending on $g$ such that $[x,{}_n\, g]\in{\mathcal E}(g)$ for every $x\in G$.

Almost Engel groups were introduced in \cite{khushu} where it was proved that an almost Engel compact group is necessarily finite-by-(locally nilpotent). The purpose of the present article is to prove the following related result.

\begin{theorem}\label{main} Let $G$ be a linear group. 
\begin{itemize}
\item[1.] If $G$ is almost Engel, then $G$ is finite-by-hypercentral.
\item[2.] If $G$ is almost nil, then $G$ is finite-by-nilpotent.
\end{itemize}
\end{theorem}

Recall that the union of all terms of the (transfinite) upper central series of $G$ is called the hypercenter. The group $G$ is hypercentral if it coincides with its hypercenter. The hypercentral groups are known to be locally nilpotent (see \cite[P. 365]{rob}). By well-known results obtained in \cite{gara,grue}, if under the hypotheses of Theorem \ref{main} the group $G$ is Engel or nil, then $G$ is hypercentral or nilpotent, respectively.

\section{Preliminaries} 

Let $G$ be a group and $g\in G$ an almost Engel element so that there is a finite set ${\mathcal E}(g)$ such that for every $x\in G$ there is a positive integer $n(x,g)$ with the property that $[x,{}_n\, g]$ belongs to ${\mathcal E}(g)$ whenever $n(x,g)\leq n$. If ${\mathcal E}'(g)$ is another finite set with the same property for possibly different numbers $n'(x,g)$, then ${\mathcal E}(g)\cap {\mathcal E}'(g)$ also satisfies the same condition with the numbers $n''(x,g)=\max\{n(x,g),n'(x,g)\}$. Hence there is a \emph{minimal} set with the above property. The minimal set will again be denoted by ${\mathcal E}(g)$ and, following \cite{khushu}, called the \emph{Engel sink for $g$}, or simply \emph{$g$-sink} for short. From now on we will always use the notation ${\mathcal E}(g)$ to denote the (minimal) Engel sinks. In particular, it follows that for each $x\in {\mathcal E}(g)$ there exists $y\in {\mathcal E}(g)$ such that $x=[y,g]$. More generally, given a subset $K\subseteq G$ and an almost Engel element $g\in G$, we write $\mathcal E(g,K)$ to denote the minimal subset of $G$ with the property that for every $x\in K$ there is a positive integer $n(x,g)$ such that $[x,{}_n\, g]$ belongs to ${\mathcal E}(g,K)$ whenever $n(x,g)\leq n$. Throughout the article we use the symbols $\langle X\rangle$ and $\langle X^G\rangle$ to denote the subgroup generated by a set $X$ and the minimal normal subgroup of $G$ containing $X$, respectively. 

A group is said to virtually have certain property if it contains a subgroup of finite index with that property. The following lemma can be found in \cite[Ch. 12, Lemma 1.2]{pasman} or in \cite[Lemma 21.1.4]{km}.

\begin{lemma}\label{passma} A virtually abelian group contains a characteristic abelian subgroup of finite index.

\end{lemma}

As usual, we write $Z_i(G)$ for the $i$th term of the upper central series of $G$ and $\gamma_i(G)$ for the $i$th term of the lower central series. Well-known Schur's theorem says that if $G$ is central-by-finite, then the commutator subgroup $G'$ is finite (see \cite[10.1.4] {rob}). Baer proved that if, for a positive integer $k$, the quotient $G/Z_k(G)$ is finite, then so is $\gamma_{k+1}(G)$ (see \cite[14.5.1] {rob}). Recently, the following related result was obtained in \cite{degiova} (see also \cite{kurda}).

\begin{theorem}\label{sysak} Let $G$ be a group and let $H$ be the hypercenter of $G$. If $G/H$ is finite, then $G$ has a finite normal subgroup $N$ such that $G/N$ is hypercentral.
\end{theorem}

We will also require the Dicman Lemma (see \cite[14.5.7]{rob}).

\begin{lemma}\label{dietz} In any group a normal finite subset consisting of elements of finite order generates a finite subgroup.
\end{lemma}

In \cite{Pl} Plotkin proved that if a group $G$ has an ascending series whose quotients locally satisfy the maximal condition, then the Engel elements of $G$ form a locally nilpotent subgroup. In particular we have the following lemma.
\begin{lemma}\label{plotkin} Let $G$ be a group having an ascending series whose quotients locally satisfy the maximal condition and let $a\in G$ be an Engel element. Then $\langle a^G\rangle$ is locally nilpotent.
\end{lemma}

Linear groups are naturally equipped with the Zarisski topology. If $G$ is a linear group, the connected component of $G$ containing 1 is denoted by $G^0$. We will use (sometimes implicitly) the following facts on linear groups. All these facts are well-known and are provided here just for the reader's convenience.

\begin{itemize}
\item If $G$ is a linear group and $N$ a normal subgroup which is closed in the Zarissky topology, then $G/N$ is linear (see \cite[Theorem 6.4]{wehr}).

\item Since finite subsets of $G$ are closed in the Zarisski topology, it follows that any finite subgroup of a linear group is closed. Hence $G/N$ is linear for any finite normal subgroup $N$.

\item If $G$ is a linear group, the connected component $G^0$ has finite index in $G$ (see \cite[Lemma 5.3]{wehr}).

\item Each finite conjugacy class in a linear group centralizes $G^0$ (see \cite[Lemma 5.5]{wehr}).

\item In a linear group any descending chain of centralizers is finite. This follows from \cite[Lemma 5.4]{wehr} and the fact that the Zarissky topology satisfies the descending chain condition on closed sets.

\item A linear group generated by normal nilpotent subgroups is nilpotent (see Gruenberg \cite{grue}).

\item Tits alternative: A finitely generated linear group either is virtually soluble or contains a subgroup isomorphic to a nonabelian free group (see \cite{tits}).

\item The Burnside-Schur theorem: A periodic linear group is locally finite (see \cite[9.1]{wehr}).

\item Zassenhaus theorem: A locally soluble linear group is soluble. Every linear group contains a unique maximal soluble normal subgroup (see \cite[Corollary 3.8]{wehr}).

\item Since the closure in the Zarisski topology of a soluble subgroup is again soluble (see \cite[Lemma 5.11]{wehr}), it follows that the unique maximal soluble normal subgroup of a linear group is closed. In particular, if $G$ is linear and $R$ is the unique maximal soluble normal subgroup of $G$, then $G/R$ is linear and has no nontrivial normal soluble subgroups.

\item A locally nilpotent linear group is hypercentral (see \cite{gara} or \cite{grue}).

\item Gruenberg: The set of Engel elements in a linear group $G$ coincides with the Hirsch-Plotkin radical of $G$. The set of right Engel elements coincides with the hypercenter of $G$ (see \cite{grue}).
\end{itemize}

\noindent Here, as usual, the Hirsch-Plotkin radical of a group is the maximal normal locally nilpotent subgroup. An element $g\in G$ is a right Engel element if for each $x\in G$ there exists a positive integer $n$ such that $[g,{}_n\, x]=1$.

\section{Almost Engel elements in virtually soluble groups} 

In the present section we give certain criteria for a group containing almost Engel elements to be finite-by-nilpotent or finite-by-hypercentral. In particular, we prove that a  virtually soluble group generated by finitely many almost Engel elements is finite-by-nilpotent (Theorem \ref{ccc}).

\begin{lemma}\label{zero} Let $G=H\langle a_1,\dots,a_s\rangle$, where $H$ is a normal subgroup and $a_i$ are almost Engel elements. Assume that $G/H$ is nilpotent. If $N\leq H$ is a finite normal subgroup of $H$, then $\langle N^G\rangle$ is finite.
\end{lemma}
\begin{proof} Suppose first that $s=1$ and write $a$ in place of $a_1$. Let $M$ be the subgroup generated by all commutators of the form $[x,{}_j\, a]$, where $x\in N$ and $j$ is a nonnegative integer. Since both $N$ and $\mathcal E(a)$ are finite, it follows that there exists an integer $k$ such that $M$ is contained in the product $\prod_{i=0}^kN^{a^i}$. It is clear that the product $\prod_{i=0}^kN^{a^i}$ is normal in $H$ and $a$ normalizes $M$. Therefore $\langle M^H\rangle$ is normal in $G$ and is contained in $\prod_{i=0}^kN^{a^i}$. Moreover, $\langle N^G\rangle=\langle M^H\rangle$ so in the case where $s=1$ the lemma follows.

Therefore we will assume that $s\geq2$ and use induction on $s$. Assume additionally that $G/H$ is abelian. Set $H_0=H$ and $H_i=H_{i-1}\langle a_i\rangle$ for $i=1,\dots,s$. The subgroups $H_i$ are normal in $G$ and $H_s=G$. By induction, $K=\langle N^{H_{s-1}}\rangle$ is finite. Since $G=H_{s-1}\langle a_s\rangle$, the above paragraph shows that $\langle K^G\rangle$ is finite. Obviously, $\langle K^G\rangle=\langle N^G\rangle$ and so in the case where $G/H$ is abelian the lemma follows.

We will now allow $G/H$ to be nonabelian, say of nilpotency class $c$. We will use induction on $c$. Set $B=\langle a_s^G\rangle$ and $G_1=HB$. Since $G/H$ is a finitely generated nilpotent group, it follows that each subgroup of $G/H$ is finitely generated and so $B$ has finitely many conjugates of $a_s$, say $a_s^{g_1}\dots,a_s^{g_r}$ such that $G_1=H\langle a_s^{g_1}\dots,a_s^{g_r}\rangle$. Since $G_1/H$ has nilpotency class at most $c-1$, by induction $\langle N^{G_1}\rangle$ is finite. We now notice that $G=G_1\langle a_1,\dots,a_{s-1}\rangle$ so the induction on $s$ completes the proof.
\end{proof}

\begin{lemma}\label{cyclic} Let $G=H\langle a\rangle$, where $H$ is a virtually abelian normal subgroup and $a$ is an almost Engel element. Then $\langle a^G\rangle$ is finite-by-(locally nilpotent).
\end{lemma}
\begin{proof} Assume that $G$ is a counter-example with $|\mathcal E(a)|$ as small as possible. In view of Lemma \ref{passma} we can choose a maximal characteristic abelian subgroup $V$ in $H$. Since $V$ is abelian, we have $[v_1,a][v_2,a]=[v_1v_2,a]$ for any $v_1,v_2\in V$. In other words, a product of two commutators of the form $[v,a]$, where $v\in V$, again has the same form. Therefore $\mathcal E(a,V)$ is a finite subgroup. Obviously, the normalizer in $G$ of $\mathcal E(a,V)$ has finite index. It follows that $\mathcal E(a,V)$ is contained in a finite normal subgroup $N$. If $\mathcal E(a,V)\neq1$, we pass to the quotient $G/N$ and use induction on $|\mathcal E(a)|$. Therefore without loss of generality we will assume that $\mathcal E(a,V)=1$, that is, $a$ is Engel in $V\langle a\rangle$. Since $\mathcal E(a)$ consists of commutators of the form $[x,a]$ with $x\in \mathcal E(a)$, it follows that $\mathcal E(a)\cap V=\{1\}$. Let $C_0=1$ and $$C_i=\{v\in V\ \vert\ [v,a]\in C_{i-1}\}$$ for $i=1,2,\dots$. Since $a$ is Engel in $V$, we have $V=\cup_i C_i$.

Let $T=\langle \mathcal E(a),a\rangle$ and $U=V\cap T$. We observe that $U$ is a finitely generated abelian subgroup. In view of the fact that $V$ is the union of the $C_i$ we deduce that there exists a positive integer $n$ such that $U=C_n\cap U$.

For $i=0,\dots,n$ set $U_i=C_i\cap U$. Thus, $U=U_n$. Observe that $U_1$ centralizes $a$ and therefore $U_1$ normalizes the set $\mathcal E(a)$. Denote by $W_{1}$ the intersection $U_1\cap C_G(\mathcal E(a))$. Since $\mathcal E(a)$ is finite, it follows that $W_1$ has finite index in $U_1$. Further, it is clear that $W_1$ is contained in the center $Z(T)$. 

The finiteness of the index $[U_1:W_1]$ implies that $U_2$ contains a normal in $T$ subgroup $W_2$ such that the index $[U_2:W_2]$ is finite, and $[W_2,T]\leq W_1$. Thus, $W_2$ is contained in $Z_2(T)$, the second term of the upper central series of $T$.

Next, in a similar way we conclude that $U_3\cap Z_3(T)$ has finite index in $U_3$ and so on. Eventually, we deduce that $U\cap Z_n(T)$ has finite index in $U$. Thus, $T/Z_n(T)$ is finite-by-cyclic and therefore there exists a positive integer $k$ such that $a^k\in Z_{n+1}(T)$. Hence, $T/Z_{n+1}(T)$ is finite and so, in view of Baer's theorem, we deduce that $T$ is finite-by-nilpotent. In particular, for some positive integer $r$ the subgroup $\gamma_r(T)$ is finite. The observation that for each $x\in {\mathcal E}(a)$ there exists $y\in {\mathcal E}(a)$ such that $x=[y,g]$ guarantees that $\mathcal E(a)$ is contained in $\gamma_r(T)$. In particular, we proved that the subgroup $\langle \mathcal E(a)\rangle$ is finite. Because $V$ is abelian, it is obvious that $V$ normalizes $V\cap\langle \mathcal E(a)\rangle$. Thus, $V\cap\langle\mathcal E(a)\rangle$ is a finite subgroup with normalizer of finite index. It follows that $V\cap\langle\mathcal E(a)\rangle$ is contained in a finite normal subgroup of $G$. We can factor out the latter and without loss of generality assume that $V\cap\langle\mathcal E(a)\rangle=1$. 

Recall that $C_1=C_V(a)$. Therefore $C_1$ normalizes $\langle\mathcal E(a)\rangle$ and in view of the fact that $V\cap\langle\mathcal E(a)\rangle=1$ we conclude that $C_1$ centralizes $\langle\mathcal E(a)\rangle$. So $C_1\leq Z(VT)$. Same argument shows that $C_2/C_1\leq Z(VT/C_1)$ and, more generally, $C_{i+1}/C_i\leq Z(VT/C_i)$ for $i=0,1,2\dots$. Thus, $V\leq Z_\infty(VT)$ where $Z_\infty(VT)$ stands for the hypercenter of $T$. Of course, it follows that there exists a positive integer $k$ such that $a^k\in Z_\infty(VT)$. We deduce that $Z_\infty(VT)$ has finite index in $VT$. Theorem \ref{sysak} now tells us that $VT$ has a finite normal subgroup $N$ such that the quotient group $(VT)/N$ is hypercentral. The hypercentral groups are locally nilpotent and so $VT$ is finite-by-(locally nilpotent). The observation that for each $x\in {\mathcal E}(a)$ there exists $y\in {\mathcal E}(a)$ such that $x=[y,g]$ guarantees that $\mathcal E(a)$ is contained in $N$.

Since $VT$ has finite index in $G$, Dicman's lemma tells us that $G$ contains a finite normal subgroup $R$ such that $\mathcal E(a)\subseteq N\leq R$. The image of $a$ in $G/R$ is Engel and the required result follows from Lemma \ref{plotkin}.
\end{proof}

\begin{theorem}\label{ccc} A virtually soluble group generated by finitely many almost Engel elements is finite-by-nilpotent.
\end{theorem}
\begin{proof} Let $G$ be a virtually soluble group generated by finitely many almost Engel elements $a_1,\dots,a_s$ and let $S$ be a normal soluble subgroup of finite index in $G$. We assume that $S\neq1$ and let $V$ be the last nontrivial term of the derived series of $S$. By induction on the derived length of $S$ we assume that $G/V$ is finite-by-nilpotent. Therefore $G$ contains a normal subgroup $H$ such that $V$ has finite index in $H$ and the quotient $G/H$ is nilpotent. For $i=1,\dots,s$ set $G_i=H\langle a_i\rangle$. By Lemma \ref{cyclic} each subgroup $\langle a_i^{G_i}\rangle$ has a finite normal subgroup $N_i$ such that $\langle a_i^{G_i}\rangle/N_i$ is locally nilpotent. Since $G_i/H$ are abelian, it is clear that all quotients $G_i/H\cap N_i$ are locally nilpotent and so, replacing if necessary $N_i$ by $H\cap N_i$, without loss of generality we can assume that all subgroups $N_i$ are normal subgroups of $H$. Therefore the product of the subgroups $N_i$ is finite. By Lemma \ref{zero} the product of $N_1\cdots N_s$ is contained in a finite subgroup $N$ which is normal in $G$. Obviously the images in $G/N$ of the generators $a_1,\dots,a_s$ are Engel. Thus, $G/N$ is a virtually soluble group generated by finitely many Engel elements. It follows from Lemma \ref{plotkin} that $G/N$ is nilpotent. The proof is complete.
\end{proof}

The next lemma is well-known. For the reader's convenience we provide the proof.

\begin{lemma}\label{helpo} Let $G=H\langle a\rangle$, where $H$ is a nilpotent normal subgroup and $a$ is a nil element. Then $G$ is nilpotent.
\end{lemma}
\begin{proof} Suppose that $a$ is $n$-Engel. Let $K=Z(H)$ and set $K_0=K$ and $K_{i+1}=[K_i,a]$ for $i=0,1,\dots$. Then $K_{n-1}\leq K\cap C_K(a)$ and so $K_{n-1}\leq Z(G)$. Moreover we observe that $[K_{i-1},G]\leq K_i$ and it follows that $K_{n-i}\leq Z_{i}(G)$ for $i=1,2,\dots,n$. Therefore $K\leq Z_{n}(G)$. Passing to the quotient $G/Z_{n}(G)$ and using induction on the nilpotency class of $H$ we deduce that if $H$ is nilpotent with class $c$, then $G$ is nilpotent with class at most $cn$.
\end{proof}

\begin{lemma}\label{helpo2} Let $G=H\langle a\rangle$, where $H$ is a hypercentral normal subgroup and $a$ is an Engel element. Then $G$ is hypercentral.
\end{lemma}
\begin{proof} It is sufficient to show that $Z(G)\neq1$. Let $Z=Z(H)$. Since $a$ is an Engel element, $C_Z(a)\neq1$. Obviously, $C_Z(a)\leq Z(G)$. The proof is complete.
\end{proof}

\begin{lemma}\label{fini} Let $a$ be an almost Engel element in a group $G$ and assume that $\mathcal E(a)$ is contained in a locally nilpotent subgroup. Then the subgroup $\langle\mathcal E(a)\rangle$ is finite.
\end{lemma}
\begin{proof} Set $D=\langle\mathcal E(a)\rangle$. Without loss of generality we can assume that $G=D\langle a\rangle$. Since $\mathcal E(a)$ is finite, $D$ is nilpotent and we can use induction on the nilpotency class of $D$. Thus, by induction assume that the quotient of $D$ over its center is finite. By Schur's theorem the derived group $D'$ is finite as well. Factoring out $D'$ we can assume that $D$ is abelian. So now $D$ is abelian and $D=[D,a]$. By \cite[Lemma 2.3]{khushu}, $D=\mathcal E(a)$ and hence $D$ is finite. 
\end{proof}

\begin{lemma}\label{nilpo} Let $G=H\langle a\rangle$, where $H$ is a hypercentral normal subgroup.

\begin{itemize}
\item[1.] If $a$ is almost Engel, then $G$ is finite-by-hypercentral.
\item[2.] If $H$ is nilpotent and $a$ is almost nil, then $G$ is finite-by-nilpotent.
\end{itemize}
\end{lemma}
\begin{proof} We will prove Claim 1 first. Assume that $a$ is almost Engel. Let $N$ be the product of all normal subgroups of $G$ whose intersection with $\mathcal E(a)$ is $\{1\}$. It is easy to see that $N\cap \mathcal E(a)=\{1\}$ and $N$ is the unique maximal normal subgroup with that property. Therefore $K\cap \mathcal E(a)\neq\{1\}$ whenever $K$ is a normal subgroup containing $N$ as a proper subgroup. Since $\mathcal E(a)$ is finite, the group $G$ contains a minimal normal subgroup $M$ such that $N<M$. Taking into account that $H$ is hypercentral, we observe that $M/N$ is central in $H/N$.

Let $D=\langle \mathcal E(a)\rangle\cap M$. It follows that $M=ND$. Suppose that $D$ is not normal in $M$ and set $L=N_M(N_M(D))$. Since $M$ is hypercentral, it satisfies the normalizer condition and so $L\neq N_M(D)$. Obviously $a$ normalizes both $L$ and $N_M(D)$. Since $a$ acts on $L/N_M(D)$ as an Engel element, the centralizer of $a$ in $L/N_M(D)$ is nontrivial. Thus, $L$ has a subgroup $C$ such that $N_M(D)<C$ and $C$ normalizes $N_M(D)\langle a\rangle$. Of course, $D$ is normal in $N_M(D)\langle a\rangle$. By Lemma \ref{helpo2} the quotient of $N_M(D)\langle a\rangle$ by $D$ is hypercentral. It is easy to see that $D$ is a unique minimal normal subgroup of $N_M(D)\langle a\rangle$ whose quotient is hypercentral. Therefore $D$ is characteristic in $N_M(D)\langle a\rangle$ and so $C$ normalizes $D$. This is a contradiction since $N_M(D)<C$. 

Hence, $D$ is normal in $M$. Again, it is easy to see that $D$ is a unique minimal normal subgroup of $M\langle a\rangle$ whose quotient is hypercentral. Therefore $D$ is characteristic in $M$ and so it is normal in $G$. We pass to the quotient $G/D$ and Claim 1 now follows by straightforward induction on $|\mathcal E(a)|$. 

We now assume that $H$ is nilpotent and $a$ is almost nil. We already know that $G$ is finite-by-hypercentral. Factoring out a finite normal subgroup we can assume that $G$ is hypercentral. In that case $a$ is actually nil and so by Lemma \ref{helpo} $G$ is nilpotent. The proof of the lemma is complete.
\end{proof}

\section{Linear groups}

\begin{lemma}\label{nilpott} A virtually soluble almost Engel linear group is finite-by-hypercentral.
\end{lemma}
\begin{proof} Suppose that $G$ is a virtually soluble almost Engel linear group. Let $S$ be a normal soluble subgroup of finite index in $G$. By induction on the derived length of $S$ we assume that $S'$ is finite-by-hypercentral. Passing to the quotient over a normal finite subgroup without loss of generality we can assume that $S'$ is hypercentral. By Lemma \ref{nilpo} the subgroup $\langle S',x\rangle$ is finite-by-hypercentral for each $x\in G$. Thus, for each $x\in G$ there exists a finite characteristic subgroup $R_x\leq\langle S',x\rangle$ such that $\langle S',x\rangle/R_x$ is hypercentral. Since $\langle S',x\rangle$ is normal in $S$, it follows that each element in $R_x$ has centralizer of finite index in $S$, hence centralizer of finite index in $G$. Therefore $G^0$ centralizes $R_x$ and it follows that $\langle S',x\rangle$ is hypercentral for each $x\in G^0$. The subgroup $\prod\langle S',x\rangle$, where $x$ ranges over $S\cap G^0$, is locally nilpotent and therefore hypercentral. In particular $N=S\cap G^0$ is hypercentral and so $G$ is virtually hypercentral. By Lemma \ref{nilpo} the subgroup $\langle N,x\rangle$ is finite-by-hypercentral for each $x\in G$. In other words, for each $x\in G$ there exists a finite characteristic subgroup $Q_x\leq\langle N,x\rangle$ such that the quotient $\langle N,x\rangle/Q_x$ is hypercentral. Since $N$ has finite index in $G$, it follows that $G$ contains only finitely many subgroups of the form $\langle N,x\rangle$. Set $N_0=\prod_{x\in G}Q_x$. We see that $N_0$ is a finite normal subgroup. Pass to the quotient $G/N_0$. Now the subgroup $\langle N,x\rangle$ is hypercentral for each $x\in G$. It follows that $N$ consists of right Engel elements and so, by the result of Gruenberg, $N$ is contained in the hypercenter of $G$. It follows from Theorem \ref{sysak} that $G$ is finite-by-hypercentral, as required. 
\end{proof}

We are now ready to prove Theorem \ref{main} in its full generality. For the reader's convenience we restate it here.

\begin{theorem} Let $G$ be a linear group. If $G$ is almost Engel, then $G$ is finite-by-hypercentral. If $G$ is almost nil, then $G$ is finite-by-nilpotent.
\end{theorem}
\begin{proof} Assume that $G$ is almost Engel. In view of Lemma \ref{nilpott} it is sufficient to show that $G$ is virtually soluble. By the Zassenhaus theorem a linear group is soluble if and only if it is locally soluble. Therefore it is sufficient to show that $G$ is virtually locally soluble. It is clear that $G$ does not contain a subgroup isomorphic to a nonabelian free group. Hence, by Tits alternative, any finitely generated subgroup of $G$ is virtually soluble. Therefore, by Theorem \ref{ccc}, any finitely generated subgroup of $G$ is finite-by-nilpotent. It becomes obvious that elements of finite order in $G$ generate a periodic subgroup. Moreover, the quotient of $G$ over the subgroup generated by all elements of finite order is locally nilpotent. Hence, $G$ is virtually locally soluble if and only if so is the subgroup generated by elements of finite order. Therefore without loss of generality we can assume that $G$ is an infinite periodic (and locally finite) group.

Let $R$ be the soluble radical of $G$. We can pass to the quotient and without loss of generality assume that $R=1$. So in particular $G$ has no nontrivial Engel elements. By the theorem of Hall-Kulatilaka $G$ contains an infinite abelian subgroup \cite{haku}. We conclude that some centralizers in $G$ are infinite. Since $G$ satisfies the minimal condition on centralizers, it follows that $G$ has a subgroup $D\neq1$ such that the centralizer $C=C_G(D)$ is infinite while $C_G(\langle D,x\rangle)$ is finite for each $x\in G\setminus D$. Using that $C$ is infinite we deduce from the Hall-Kulatilaka theorem that $C$ contains an infinite abelian subgroup $A$. Obviously $A\leq C_G(\langle D,A\rangle)$ and it follows that $A\leq D$. Thus, $A\leq Z(C)$.

Now choose $1\neq a\in A$. The centralizer $C$ normalizes the finite set $\mathcal E(a)$ because $a\in Z(C)$. Hence, $C$ contains a subgroup of finite index which centralizes $\mathcal E(a)$. It follows that $C_G(\langle D,\mathcal E(a)\rangle)$ is infinite and we conclude that $\mathcal E(a)$ is contained in $D$ and $C$ centralizes $\mathcal E(a)$. In particular, $a$ centralizes $\mathcal E(a)$ and so $\mathcal E(a)=\{1\}$. Thus, $a$ is an Engel element, a contradiction. This completes the proof of Claim 1.

Suppose now that $G$ is almost nil. We already know that $G$ is finite-by-hypercentral. Passing to a quotient over a finite normal subgroup we can assume that $G$ is hypercentral. Then obviously $G$, being both hypercentral and almost nil, must be nil. By the result of Gruenberg, $G$ is nilpotent.
\end{proof}

\end{document}